\documentclass[11pt]{amsart}

\usepackage[utf8]{inputenc}
\usepackage{amsfonts, enumerate}
\usepackage{amssymb}
\usepackage{latexsym}
\newtheorem{theorem}{Theorem}[section]

\newtheorem{problem}[theorem]{Problem}
\newtheorem{lemma}[theorem]{Lemma}
\newtheorem{corollary}[theorem]{Corollary}
\newtheorem{remark}[theorem]{Remark}

\newtheorem{definition}[theorem]{Definition}
\DeclareMathOperator{\Cay}{Cay}
\newcommand{\ZZ}{\ensuremath{\mathbb{Z}}}
\newcommand{\NN}{\ensuremath{\mathbb{N}}}
\newcommand{\RR}{\ensuremath{\mathbb{R}}}

\newcommand{\tet}[1]{\ensuremath{{^{#1}p}}}
\def\Orbit{\mathcal C}
\def\Forward{\mathcal C^+}
\def\Project{\mathcal P}
\def\Linear{\mathcal L}

  \author{Bernhard Kr\"on, J\"org Lehnert, Maya Stein}
\thanks{MS was supported by grants Fapesp PQ-EX 2008/50338-0, Fondecyt 11090141}

  \title[Boundaries of \uppercase{HNN}-extensions and distortion phenomena]{Linear and projective boundaries in \uppercase{HNN}-extensions and distortion phenomena}

\begin{document}
\maketitle

\begin{abstract}
Linear and projective boundaries of Cayley  graphs were introduced in~\cite{kst} as  quasi-isometry invariant   boundaries of finitely generated groups. They consist of forward orbits $g^\infty=\{g^i: i\in \mathbb N\}$, or orbits $g^{\pm\infty}=\{g^i:i\in\mathbb Z\}$, respectively, of non-torsion elements~$g$ of the group $G$, where `sufficiently close' (forward) orbits become identified, together with a metric bounded by 1. 
\newline 
We show that for all finitely generated groups, the distance between the antipodal points $g^\infty$ and $g^{-\infty}$ in the linear boundary is bounded from below by $\sqrt{1/2}$, and we give an example of a group which has two antipodal elements of distance at most $\sqrt{12/17}<1$. Our example is a  derivation of the Baumslag-Gersten group. 
\newline
We also exhibit a group with elements $g$ and $h$ such that $g^\infty = h^\infty$, but $g^{-\infty}\neq h^{-\infty}$. Furthermore, we introduce a notion of average-case-distortion---called growth---and compute explicit positive lower bounds for distances between points $g^\infty$ and $h^\infty$ which are limits of group elements $g$ and $h$ with different growth.
\end{abstract}

\section{Introduction}

One of the most important classes of groups studied  in Geometric Group Theory is the class of word-hyperbolic groups  (also referred to as Gromov-hyperbolic groups). Word-hyperbolic groups admit several geometric tools which can be used to derive algebraic properties.
Since in Geometric Group Theory the focus lies on the large-scale geometry of the group, these tools are only defined up to quasi-isometries. An important large-scale invariant of  a hyperbolic group is its Gromov-boundary. The present work is part of a program to understand up to which extent one can generalize this concept to arbitrary finitely generated groups. 

A new concept of quasi-isometry invariant boundaries of metric spaces has recently been introduced by Kr\"on, Lehnert, Seifter and Teufl~\cite{kst}. It is related to a concept due to Bonnington, Richter and Watkins~\cite{bonnington07between}. This concept is rather general and for instance, Tits' boundary of a \uppercase{CAT}$(0)$ space (see~\cite[Section 9]{bridson99}) fits into it, after a small modification. See~\cite{kst} for a more detailed discussion of this relationship.

We will not recall the full concept for metric spaces, because here, we are only interested in two applications to Cayley graphs of finitely generated groups, namely the linear and the projective boundary, which we shall introduce next.

Let $G$ be a group generated by a set $X$. The Cayley graph $\Gamma=(V,E)=\Cay(G,X)$ is the graph with vertex set $V=G$ and edge set $E=\{\{g,h\}: g^{-1}h\in X\}$. Let $d$ be the graph metric of $\Gamma$. That is, $d(g,h)$ is the length of the shortest path in $\Gamma$ from $g$ to $h$.

For $g\in G$ of infinite order let $g^\infty:=\{g^n:n\in \NN\}$ denote the cyclic subsemigroup generated by $g$. We also call  $g^\infty$ the {\em forward orbit} of $g$. Let $g^{\pm\infty}:=\{g^k:k\in \ZZ\}$ denote the cyclic subgroup generated by $g$, and we call $g^{\pm\infty}$ the {\em orbit} of $g$. The {\em backward orbit} $g^{-\infty}$ is defined analogously. 

Let $\Orbit G$ and $\Forward G$ denote the family of infinite orbits or infinite forward orbits, respectively. That is, we set
\[
\Orbit G:=\{g^{\pm\infty}: g\in G,\ |g|=\infty\}
\]
 and  
 \[
 \Forward G:=\{g^\infty: g\in G,\ |g|=\infty\}.
 \]
We want to measure the distance between two orbits as if it were an angle. For this, fix $\alpha>0$ and $c\in\mathbb N$, and call the set
\[
\alpha\cdot g^\infty+c:=\{v\in G: \exists n\in\mathbb N \text{ such that } d(v,g^n) \leq \alpha\cdot d(1, g^n)+c\}
\]
the \emph{$(\alpha,c)$--cone} around $g^\infty$. In other words, the \emph{$(\alpha,c)$--cone} around $g^\infty$ is the union of all balls with center $g^n$ and radius $\alpha \cdot d(1,g^n)+c$. Analogously we define the $(\alpha,c)$-cone around $g^{\pm\infty}$ as
\[
\alpha\cdot g^{\pm\infty}+c:=\{v\in G: \exists k\in\mathbb \ZZ \text{ so that } d(v,g^k) \leq \alpha \cdot d(1, g^k)+c\}.
\]
We write $h^\infty \in \alpha\cdot g^\infty+c$ if $h^n\in \alpha\cdot g^\infty+c$ for all $n\in\mathbb N$ and define $h^{\pm\infty} \in \alpha\cdot g^{\pm\infty}+c$ analogously. For $x,y\in\Orbit G$ or $x,y\in\Forward G$ set
\[
s_X(x,y):=\inf \{\alpha\in\mathbb R: \exists c\in\mathbb N\text{ such that }x\in \alpha\cdot y+c\text{ and }y\in \alpha\cdot x+c\}.
\]
If $s_X(x,y)=0$ then we call $x$ and $y$ \emph{linearly equivalent}, this is an equivalence relation. We call two elements $g$ and $h$  \emph{ forward equivalent} if  $g^\infty\sim h^\infty$ and \emph{backward equivalent} if  $g^{-\infty}\sim h^{-\infty}$.

It is easy to check that the function $s_X$ is well defined on the set of equivalence classes 
and that the square root $t_X=\sqrt s_X$ is a metric on the quotient $\Forward G/\!\!\sim$ and on $\Orbit G/\!\!\sim$, respectively. The completion of the metric space $(\Forward G/\! \!\sim,t)$ is called the  {\em linear boundary} $\Linear G$ of $G$, the completion of the metric space $(\Orbit G/\!\!\sim,t_X)$ is called the {\em projective boundary}  $\Project G$ of $G$, or strictly speaking of $G$ with respect to the generating set $X$. 
Although the elements of the linear/projective boundary are equivalence classes of (forward) orbits $g^{(\pm)\infty}$, and not the (forward) orbits themselves, we shall slightly abuse notation and write $g^{(\pm)\infty}$ instead of $[g^{(\pm)\infty}]_\sim$ also for an element of the linear or projective boundary.

If $G$ is finitely generated and we change the finite set of generators then the resulting quotient spaces are bi-Lipschitz equivalent and hence the boundaries are homeomorphic. But  the values of $s_X$ and $t_X$ depend on the choice of generators. In most cases it will be clear out of context with respect to which set of generators we calculate $s_X$ and $t_X$, therefore, we will frequently suppress the index $X$. Moreover, by definition it is clear that the diameter of  $\Linear G$ and of $\Project G$ is at most~$1$. For more details we refer to \cite{kst}.

\medskip

The linear boundary of finitely generated nilpotent groups is (homeomorphic to) the disjoint union of spheres with dimensions $d_i$, which correspond to the free abelian quotients of rank $d_i+1$ in the central series, and  the projective boundary is (homeomorphic to) the disjoint union of projective spaces of the same dimension; see~\cite{kst}.
The latter fact relies on the observation that in the case of a nilpotent group the distance $t(g^\infty,h^\infty)$ equals the distance of the inverse elements $t(g^{-\infty},h^{-\infty})$ for all $g^\infty,h^\infty\in\Linear G$. Thus the space $\Project G$ can be obtained identifying each element with its inverse without changing distances (that is, for all $g,h\in G$ the distance $t(g^{\pm\infty},h^{\pm\infty})$ (in $\Project G$) equals the minimum of $t(g^\infty,h^\infty)$ and $t(g^\infty,h^{-\infty})$ (in $\Linear G$)).

One might guess that this yields a general method to construct the projective boundary but the results in Section~\ref{sect:g-ginv} show that this is not the case. In general it is not even true that $g^\infty=h^\infty$ implies $g^{-\infty}=h^{-\infty}$ hence the projective boundary is not necessarily a quotient of the linear boundary.

\begin{theorem}\label{forwibackwi}
There is a group $H$ with elements $g_1$ and $g_2$ which are forward-equivalent but not backward-equivalent.
\end{theorem}

The proof of Theorem~\ref{forwibackwi} is given in Section~\ref{sect:g-ginv}.

\medskip

Knowing of this counterintuitive phenomenon,
 it is natural to ask whether the \lq algebraic antipodal\rq{} $g^{-\infty}$ of $g^\infty\in\Forward G$ is also the metric antipodal. In other words, one would like to know whether $t(g^\infty,g^{-\infty})$ is always $1$ or if at least this distance is universally bounded away from~$0$. We show that the answer to the first question is negative, but that there is a positive lower bound for $t(g^\infty,g^{-\infty})$.

\begin{theorem}\label{thm:distances}\ 
\begin{enumerate}[(a)]
\item For any finitely generated group $G$ and any
 $g\in G$ of infinite order we have $t(g^\infty,g^{-\infty})\geq \sqrt {1/2}$.
\item There exists a group $G$ generated by the finite set $X$ which has an element  $g$ such that $t_X(g^\infty,g^{-\infty})\leq \sqrt{12/17}$.
\end{enumerate}
\end{theorem}
The proof of this result will span from Section~\ref{sect:forward_back} to Section~\ref{sec:final}.
While the proof of the first part of  Theorem~\ref{thm:distances} is not overly complicated, the proof of the second part is quite lengthy and takes up most of these three sections in which we give an example of a family of such groups. The groups in question are derivations of the so called Baumslag-Gersten group and in order to prove our theorem we have to understand some of the intrinsic geometry of these groups. Note that for the group constructed for the second part of the statement it is not hard to see that for all $g\in G$ it holds: $\max_{h\in G}t_X(g^\infty,h^\infty)=1$. This remark goes back to a suggestion of an anonymous referee of this paper and actually it sounds reasonable that this statement is true for all finitely generated groups $G$ but we have not been able to prove it, yet.

\medskip

As we will see, the geometry of a cyclic subgroup can be very different from the usual geometry of the group of integers. This phenomenon is known as distortion and leads to one of the asymptotic invariants studied by Gromov in his seminal book~\cite{Gro}. For an element $h$ of a group $G$ generated by the finite set $X$ let $|h|_X$ denote  the length of the shortest word representing $h$ in letters of $X^\pm$, where $X^\pm=\{x\in G: x\in X\mbox{ or } x\in X^{-1}\}$. Gromov defines the distortion function for a subgroup $H$ generated by the finite set $Y$ as:
\[
\Delta_G^H(r):=\frac 1r\max\{|h|_Y : h\in H,\ |h|_X\leq r\}.
\]
This function measures something like a worst-case distortion and can easily be superexponential, for instance in the group $G_p$ of Theorem~\ref{thm:small_distance}. Such examples suggest that the factor $1/r$ is a bit artificial and in fact nowadays most authors follow the definition of Farb~\cite{farb} who defined the distortion function just as $\Delta_G^H(r):=\max\{|h|_Y : h\in H,\ |h|_X\leq r\}$.

In the context of this work we are interested in the distortion of cyclic subgroups (or even cyclic subsemigroups). But as we would like to view these subgroups just as a set rather than as a sequence, worst-case considerations do not seem appropriate.  
A better fitting concept will be a kind of average-case distortion for cyclic subgroups---called growth of elements---which we define as follows:

\begin{definition}
  Let $G$ be a group generated by the finite set $X$ and let $g\in G$. The \emph{growth} of $g$ is the function $w_g(n):\NN\rightarrow \NN$ which counts the number of elements of the type $g^i$ in the ball $B_1(n)$ of radius $n$ around $1$:
\[
w_g(n):=|\{i\in\mathbb Z:|g^i|_X\leq n\}|.
\]
\end{definition}

\medskip

Note that for the group $H=\langle g\rangle$ our growth function $w_g(n)$ measures the number of elements of $H$ in the ball of radius $r$ around $1$, while Gromov's distortion $\Delta_G^H(r)$ determines the absolute value of the maximum of all $i$ such that $g^i$ still lies in this ball. 

There are some easy bounds on the growth.
First of all, balls in Cayley graphs grow at most exponentially fast. Namely, it is easy to see that the upper bound $w_g(n)\leq|B_1(n)|\leq (2|X|-1)(2|X|)^{n-1}$ holds. 
Less obvious but still straight-forward is the following fact. For all $k\in\NN$, we have
\[
w_g(kn)\geq k\cdot w_g(n).
\]
For instance, the groups which will be defined in Theorem~\ref{thm:small_distance} contain elements with exponential growth function,  and in free nilpotent groups of class $c$ the growth function of a central element is equivalent to $n^c$. The results of Olshanskii and Sapir~\cite{OlsSap} on length functions of subgroups, which are a very precise measure for distortion phenomena,
suggest that there exist a broad variety of growth functions for elements. It seems natural to ask the following question:

\begin{problem}\label{prob:dist}
Can two elements $g$ and $h$ of a group, whose forward orbits are linearly equivalent, have growth functions of different order?
\end{problem}

In Section~\ref{sect:distort} we will give a partial solution to this problem. If  $g$ is an element of exponential growth,
then there is even a minimal distance between $g^{\pm\infty}$, and any other orbit of $\Project G$ of an element $h$ of the group which has a different growth. 
This minimal distance depends on the number of generators of $G$ and the growth functions of $g$ and $h$. Our lower bound also holds for the minimal distance in $\Linear G$.

To make this statement more precise we will use Landau notation. Recall that for a function $f,g:\NN\rightarrow\NN$ the notation $f(n)\in\omega(g(n))$ can be translated to $\forall k>0\ \exists n_0$ such that $\forall n>n_0$ holds: $f(n)\geq k\cdot g(n)$.  
In the same manner $f(n)\in o(g(n))$ translates to $\forall k>0\ \exists n_0$ such that $\forall n>n_0$ holds: $f(n)\leq k\cdot g(n)$.
\begin{theorem}\label{lem:mindist}
For every $d\in\mathbb N$, $\delta >1$ and $\gamma >\delta$ there is a $t_{\min}=t_{\min}(d,\gamma,\delta)>0$ such that for each group $G$, each generating set $X$ of cardinality $d$, and any $g,h\in G$ with $w_g(n)\in \omega(\gamma^n)$ and $w_h(n)\in o(\delta^n)$ we have that  
$$t(g^{\pm\infty}, h^{\pm\infty})\geq t_{\min}\ \text{ and } \ t(g^{\infty}, h^{\infty})\geq t_{\min}.$$
A possible choice is $t_{\min}=\sqrt{\log_{(2d-1)\gamma}{\frac{\gamma}{\delta}}}$.
\end{theorem}
Note that the assumption $w_g(n)\in \omega(\gamma^n)$ already implies that $d\geq 2$ and therefore the logarithm is well defined.

In order to be able to speak of the growth of an element of a group without fixing a generating set,
we consider equivalence classes of growth functions rather than explicit functions. Functions $f,g:\NN\rightarrow\NN$ are called \emph{weakly equivalent} if there exist constants $c_1,c_2$ such that 
\begin{eqnarray*}
g(n)&\leq& c_1f(c_1n+c_2)+c_2 \textrm{ and}\\
 f(n)&\leq& c_1g(c_1n+c_2)+c_2
\end{eqnarray*}
hold. If $X$ and $Y$ are finite generating sets for $G$, then $\Cay(G,X)$ and $\Cay(G,Y)$ are bi-Lipschitz equivalent and therefore the growth function of $g$ with respect to $X$ and the growth function of $g$ with respect to $Y$ are weakly equivalent. Note that this equivalence separates exponential functions from sub-exponential functions and hence having an exponential growth function is a property of the group element which is independent of the chosen generating set.

We say that an element of a finitely generated group has \emph{exponential growth} if there is a finite generating set $S$ of $G$ such that the growth function of $g$ with respect to $S$ is exponential  (by the preceding paragraph, this then holds for any finite generating set $S$).
Theorem~\ref{lem:mindist} immediately gives the following corollary.

\begin{corollary}\label{cor:mindist}
If $g$ is an element of a finitely generated group that has exponential growth, then every element $h$ with $g^{\infty}=h^{\infty}$ (or with with $g^{\pm\infty}=h^{\pm\infty}$) also has exponential growth.
\end{corollary}

\medskip

Before we start let us fix some further notation. Throughout the paper $G$ will be a group generated by a (usually finite) set $X$. The free monoid over the alphabet $X^\pm$ will be denoted $X^*$ and $\ell$ is the length function on $X^*$. The assumption that $X$ is a generating set of $G$ implies the existence of a surjective monoid homomorphism $\pi:X^*\rightarrow G$ and it is straightforward that for $g,h\in G$ we have
\[d(g,h)=\min\{\ell(w): w\in X^*,\ \pi(x)=g^{-1}h\}.\]
Using this fact, we mostly work with representing words for group elements. We will use the shorthand notation $w_1=_Gw_2$ for $\pi(w_1)=\pi(w_2)$ whereas $w_1=w_2$ means that the two words as elements of $X^*$ are equal.

For $Y\subset G$ we will denote by $\langle Y\rangle_G$ the subgroup of $G$ generated by $Y$, i.e. the smallest subgroup of $G$ containing $Y$ and by $\langle\langle Y\rangle\rangle_G$ the normal closure of $Y$ in $G$, i.e. the smallest normal subgroup of $G$ containing $Y$.

Beginning with Section~\ref{sect:forward_back} we will have to work with huge powers. We will use the following notation: Let $^n p$ denote the tower of length $n$ of $p$th powers (often called tetration of $p$ by $n$), i.e.~$^0 p=1$ and $^n p = p^{^{(n-1)}p}$. So for instance $^3p=p^{p^p}$. (Note that by convention $a^{b^c}=a^{(b^c)}$, not $(a^b)^c$.)

We assume that the reader is familiar with the concept of \uppercase{HNN}-extensions and in particular with Britton's Lemma which most of our considerations concerning Part~(b) of Theorem~\ref{thm:distances} rely on.  Britton's Lemma can be used to derive a normal form for elements in \uppercase{HNN}-extensions and gives a necessary condition for a word to represent the identity. The standard references for these results (and many other facts on \uppercase{HNN}-extensions) are~\cite{lynschu} and~\cite{serre}.

\section{Distortion phenomena}\label{sect:distort}

The present section is dedicated to  the aforementioned distortion phenomena. We prove Theorem~\ref{lem:mindist}.

\begin{proof}[Proof of Theorem~\ref{lem:mindist}]
We will only show the result for the  elements of the projective boundary, that is, we show the existence of a number $t_{\min}$ such  that for each group $G$ that is generated by $d$ elements,  and any $g,h\in G$ with $w_g(n)\in \omega(\gamma^n)$ and $w_h(n)\in o(\delta^n)$,  the inequality $t(g^{\pm\infty}, h^{\pm\infty})\geq t_{\min}$ holds. The other part can be shown analogously.

We assume that $t(g^{\pm\infty}, h^{\pm\infty})<1$, since otherwise $1$ is the desired bound.

  Since  $w_g(n)\in \omega(\gamma^n)$ and $w_h(n)\in o(\delta^n)$ there exist constants $N_0,c_1,c_2$, such that for all $n>N_0$ it holds:
\begin{equation}\label{growthgh}
  w_g(n)\geq c_1\cdot \gamma^n \textrm{ and }
  w_h(n)\leq c_2\cdot  \delta^n
\end{equation}
Let $n>N_0$, let $\alpha\in \mathbb{R}$ s.t. $1 >\alpha >t(g^{\pm\infty}, h^{\pm\infty})^2=s(g^{\pm\infty}, h^{\pm\infty})$. 
 
By definition there exists a constant $c$ such that for all $i\geq 0$ there exists a $j=j(i)$ such that $$g^i\in B_{\alpha d(1,h^j) +c}(h^j).$$ If $d(1,g^i)\leq n$ then by the triangle-inequality, $$d(1,h^j)\leq d(1,g^i)+d(g^i,h^j)\leq n+\alpha d(1,h^j) +c$$ and thus $d(1,h^j)\leq \frac{n+c}{1-\alpha}$.

Set $I:=\{i\in\mathbb Z: d(1,g^i)<n\}$, and set $J:=\{j\in\mathbb Z:d(1,h^j)\leq \frac{n+c}{1-\alpha}\}$. Then for each $i\in I$ we have $j(i)\in J$.
By~\eqref{growthgh}, $|I|\geq c_1\gamma^n$ and $|J|\leq c_2\delta^{\frac{n+c}{1-\alpha}}$, and the latter is smaller than $c_3\delta^{\frac n{1-\alpha}}$ for some constant $c_3$.
Hence, by the pigeon-hole principle, there exists a $j\in J$, such that 
\[
|B_{\alpha d(1,h^j) +c}(h^j)|\geq \frac{c_1\cdot \gamma^n}{c_3\cdot \delta^{\frac{n}{1-\alpha}}}=\frac{c_1}{c_3}\left(\frac{\gamma}{\delta^{\frac{1}{1-\alpha}}}\right)^n.
\]
On the other hand $|B_{\alpha d(1,h^j) +c}(h^j)|$ is bounded above by a power of the number of generators $d$, namely by
\[|B_{\alpha d(1,h^j) +c}(h^j)|\leq 2d\cdot (2d-1)^{\alpha d(1,h^j) +c-1}.\]
We obtain the inequality
\begin{eqnarray*}
  \frac{c_1}{c_3}\left(\frac{\gamma}{\delta^{\frac{1}{1-\alpha}}}\right)^n&\leq&2d\cdot (2d-1)^{\alpha d(1,h^j) +c-1}\\
&\leq&2d\cdot (2d-1)^{\alpha \frac{n+c}{1-\alpha} +c-1}\\[10pt]
&=&c_4\cdot (2d-1)^{\frac{\alpha n}{1-\alpha}}\\[10pt]
&=&c_4\cdot \left((2d-1)^{\frac{\alpha}{1-\alpha}}\right)^n,
\end{eqnarray*}
for $c_4=2d\cdot (2d-1)^{\frac{\alpha c}{1-\alpha}+c-1}$. This has to be true for arbitrary large values of $n$, which is possible only if 
\[\begin{array}{llcl}
&\frac{\gamma}{\delta^{\frac{1}{1-\alpha}}}&\leq& (2d-1)^{\frac{\alpha}{1-\alpha}}\\[10pt]
\Leftrightarrow& \gamma^{1-\alpha}&\leq&(2d-1)^\alpha\cdot\delta\\[10pt]
\Leftrightarrow& \ln{\gamma}-\alpha\cdot\ln{\gamma}&\leq& \alpha\cdot\ln{(2d-1)}+\ln{\delta}\\[10pt]
\Leftrightarrow& \frac{\ln{\gamma}-\ln{\delta}}{\ln{(2d-1)}+\ln{\gamma}}&\leq&\alpha\\[10pt]
\Leftrightarrow&\log_{(2d-1)\gamma}{\frac{\gamma}{\delta}}&\leq&\alpha.
\end{array}
\]
Note that $\frac{\gamma}{\delta}<(2d-1)\gamma$ and therefore this lower bound is less than $1$. We obtain the lower bound $t(g^{\pm\infty}, h^{\pm\infty})\geq\sqrt{\log_{(2d-1)\gamma}{\frac{\gamma}{\delta}}}$.
\end{proof}

The complete answer to Problem~\ref{prob:dist} remains open. In addition it might be an interesting project to completely understand the relationship between the usual distortion of cyclic subgroups and the growth of the generating element. It obviously happens that cyclic subgroups of different distortion yield elements of the same growth type but whether it can also be the other way around is an open question.

\section{Forward- vs. backward-equivalence}\label{sect:g-ginv}

In this section we will construct a group $H$ that contains elements $g_1$ and $g_2$ for which  $g_1^\infty\sim g_2^\infty$ but $g_1^{-\infty}\not\sim g_2^{-\infty}$. The group $H$ is an iterated HNN-extension of a cyclic group (generated by the element $a$) with stable letters $s,t,x$  given by the presentation
\begin{eqnarray}\label{presentation}
H&=&\left\langle a,s,t,x \mid t^{-1}at=a^2,s^{-1}as=a^2,x^{-1}sx=s^2\right\rangle.
\end{eqnarray}
Thus $H$ is isomorphic to a free product with amalgamation $H=H_1\ast_{\langle a\rangle} H_2$ where $H_1$ is the Baumslag-Solitar group $BS(1,2)=\langle a,t\mid t^{-1}at=a^2\rangle$ and $H_2=\langle a,s,x\mid s^{-1}as=a^2,x^{-1}sx=s^2\rangle$ is an HNN-extension of $BS(1,2)=\langle a,s\mid s^{-1}as=a^2\rangle$ with associated subgroups $\langle s\rangle$ and $\langle s^2\rangle$. 

We use the group $H$ to prove Theorem~\ref{forwibackwi}.

\begin{proof}[Proof of Theorem~\ref{forwibackwi}]
We have to show that  $H$ contains elements $g_1$ and $g_2$ which are forward-equivalent but not backward-equivalent. We do this for  $g_1:=t$ and $g_2:=at$. 

First of all, we estimate the distance $d_H(1,g_i^k)$ for $k\in\ZZ$. In all defining relations of presentation~(\ref{presentation}) the exponent sum of $t$ is zero, hence any word representing $t^k$ needs at least $|k|$ times the letter $t$ (or $t^{-1}$ if $k<0$). So the word $t^k$ is geodesic and 
\begin{equation}\label{distg1}
 d(1,g_1^k)=|k|.
\end{equation}
 The same argument yields that 
\begin{equation}\label{distg2}
|k|\leq d(1,g_2^k)\leq 2|k|,
\end{equation}
 which will be a sufficient approximation for our purpose.

Let $k>0$. We can use the relation $t^{-1}at=a^2$, which is the same as $at=ta^2$, to see that 
\begin{equation}\label{g2}
 g_2^k=t^ka^{2^{k+1}-2}.
\end{equation}
By definition, the distance $d_H(g_1^k,g_2^k)$ is  the same as 
\begin{equation}\label{11-1}
d_H(1,g_1^{-k}g_2^k)=d_H(1,t^{-k}t^ka^{2^{k+1}-2})=d_H(1,a^{2^{k+1}-2}).
\end{equation}
 One easily checks that 
 \begin{equation}\label{11-2}
 a^{2^{k+1}-2}=s^{-(k+1)}as^{k+1}a^{-2}.
 \end{equation}
Hence to obtain an upper bound for $d_H(1,g_1^{-k}g_2^k)$ we need to find a good upper bound for $d_H(1,s^k)$. Let $k_mk_{m-1}\ldots k_0$ be the binary code for $k$ (that is, $k_i\in\{0,1\}$ and $k_m=1$). Then, because of the relation $x^{-1}sx=s^2$, it holds that  $(\prod_{i=0}^{m-1} s^{k_i}x^{-1})sx^m=s^k$. The fact that $m=\lfloor\log_2 k\rfloor$  gives us the upper bound $d_H(1,s^k)\leq 3\cdot\lfloor\log_2 k\rfloor+1$. Thus by~\eqref{11-1} and ~\eqref{11-2},
 \begin{equation}\label{11-3}
 d_H(1,g_1^{-k}g_2^k)\leq 6\cdot\lfloor\log_2 (k+1)\rfloor+5.
\end{equation}

In order to show that $d^X_{\mathcal{LH}}(g_1^\infty,g_2^\infty)=0$, we now fix an $\alpha>0$ and show that $d_{\mathcal{L}H}(g_1^\infty,g_2^\infty)\leq\alpha$. To do so, by~\eqref{distg2}, it suffices to show that there exists a constant $c=c(\alpha)$ such that for each $k$ there exist $k_1$ and $k_2$ such that $d_H(g_1^k,g_2^{k_1})<\alpha\cdot k_1+c$ and $d_H(g_2^k,g_1^{k_2})<\alpha\cdot k_2+c$. Choosing $k_1=k_2=k$ and using~\eqref{11-3}, this breaks down to the statement that there exists a constant $c=c(\alpha)$ such that 
\[6\cdot\lfloor\log_2 (k+1)\rfloor+5\leq\alpha\cdot k+c,\]
which is obviously true. This shows that $g_1$ and $g_2$ are forward-equivalent.

We shall now show that $g_1$ and $g_2$ are not backward-equivalent. In fact, we claim that $d_{\mathcal{L}H}(g_1^{-\infty},g_2^{-\infty})=1$. 
For this, by~\eqref{distg1}, it suffices to show that for each $c\in\mathbb N$ there exists an $l'\in\mathbb N$ such that for all $l\in\mathbb N$ the inequality
\[d_H(g_2^{-l'},g_1^{-l})>1\cdot d(1,g_1^l)  +c=l+c\]
holds. Set $l':=c+2$.
By definition, and because of the relation $t^{-1}a^{-1}t=a^{-2}$, we have
\[
d_H(g_2^{-l'},g_1^{-l})=d_H(1,g_1^lg_2^{-l'})=d_H(1,t^la^{-(2^{l'+1}-2)}t^{-l'}),
\]
where for the last equality we used~\eqref{g2}.

Now, let $h=t^la^{-(2^{l'+1}-2)}t^{-l'}$ be the word representing $g_1^lg_2^{-l'}$ in $H$ and we try to simplify it within the presentation of this group. Using $2^{l'}-1$ times the relation  $t^{-1}a^{-1}t=a^{-2}$ we obtain that $h=_Ht^{l-1}a^{-2^{l'}+1}t^{-l'+1}=:h'$. In order to give a lower bound for $d_H(1,h)$ we once again have to change our point of view. The group $H$ is an HNN-extension of $H_2$ with stable letter $t$ and associated subgroups $\langle a\rangle$ and $\langle a^2\rangle$. Let $\overline{h}$ be a geodesic word such that $h'=_H\overline{h}$. Hence $h'\overline{h}^{-1}=_H1$. We now iteratively apply Britton's Lemma to $h'\overline{h}^{-1}$. 

The number of $a^{-1}$'s in $h'$ is odd and therefore the $t$'s and $t^{-1}$'s belonging to $h'$ cannot cancel out (moving a $t$ from left to right through a power of $a$'s halves this power). Therefore they all have to cancel with corresponding $t^{-1}$ and $t$ letters in $\overline{h}^{-1}$. This implies that the geodesic word $\overline{h}$ has to contain $l-1$ times the letter $t$ and $l'-1$ times the letter $t^{-1}$. So,  

\[
d_H(g_1^{-l},g_2^{-l'})=d_H(1,h)\geq (l-1)+(l'-1)> 1\cdot l +c,
\]
 as desired.
\end{proof}

\section{The distance between $g^{\infty}$ and $g^{-\infty}$}\label{sect:forward_back}

The remainder of this paper is devoted to the proof of Theorem~\ref{thm:distances}. We split it into two parts. First we show in 
 Theorem~\ref{thm:larger_distance}  the easier lower bound for the distance  between two elements $g^{\infty}$ and $g^{-\infty}=(g^{-1})^\infty$ of the linear boundary of a finitely generated group $G$.  The more difficult part of Theorem~\ref{thm:distances} is obtained from Theorem~\ref{thm:small_distance}, which shows that there are examples of groups with elements $g$ where the distance between $g^{\infty}$ and $g^{-\infty}$ is strictly smaller than $1$. The proof of Theorem~\ref{thm:small_distance} will continue in Sections~\ref{sec:lemmas} and~\ref{sec:final}.
 
 But let us first show the easier bound:
 
\begin{theorem}\label{thm:larger_distance}
Let $g$ be an element of a finitely generated group of infinite order. Then  $t(g^\infty,g^{-\infty})\geq 1/\sqrt 2$.
\end{theorem}
\begin{proof}
Any ball in a group with respect to a finite generating set is finite. Hence 
\begin{equation}\label{22}
 \lim_{i\to\infty}d(1,g^{i})=\infty.
\end{equation}
Suppose $\alpha\in\RR$ is such that $s(g^\infty,g^{-\infty})<\alpha$. 
Then there is a $c\in\NN$ such that for each $i$ there exists an $m(i)\in \mathbb N$ with
\begin{eqnarray*}
d(g^{-i},g^{m(i)})&\leq&\alpha\cdot d(1,g^{m(i)})+c\\
&\leq&\alpha\cdot\left(d(1,g^{-i})+d(g^{-i},g^{m(i)})\right)+c,
\end{eqnarray*}
using the triangle-inequality. 
By \eqref{22}, there is an increasing sequence $(i_n)_{n\ge 1}$ such that
\begin{equation}\label{22b}
d(1, g^k)>d(1, g^{i_n})
\end{equation}
for all $k>i_n$. Thus
\begin{align*}
\alpha\ \geq & \ \frac{d(g^{-i_n},g^{m(i_n)})-c}{d(1,g^{-i_n})+d(g^{-i_n},g^{m(i_n)})}\\
= &\ \frac{d(1,g^{i_n+m(i_n)})-c}{d(1,g^{i_n})+d(1,g^{i_n+m(i_n)})}\\
\stackrel{\eqref{22b}}
{\geq} & \ \frac{d(1,g^{i_n+m(i_n)})-c}{2\cdot d(1,g^{i_n+m(i_n)})} \\
=& \ \frac 12 \left(1- \frac{c}{d(1,g^{i_n+m(i_n)})}\right).
\end{align*}
Since this inequality is valid for all $i_n$, $n\in\mathbb N$, and because of~\eqref{22}, we obtain that $\alpha\ge 1/2$. As $\alpha$ may be chosen arbitrarily close to  $s(g^\infty,g^{-\infty})$, this implies that $s(g^\infty,g^{-\infty})\ge 1/2$, and thus, $t(g^\infty,g^{-\infty})\ge 1/\sqrt 2$.
\end{proof}

We now turn to the  rather tedious proof of the second part of Theorem~\ref{thm:distances} which will span over the remainder of this section and the following two sections. 


\begin{theorem}\label{thm:small_distance}
Let $p\geq 20$. In the group $G_p=\langle a,t\mid t^{-1}a^{-1}tat^{-1}at=a^p\rangle$ it holds that $t(a^\infty,a^{-\infty})\leq \sqrt{12/17}$.
\end{theorem}

\begin{remark}
 The group $G_p$ from Theorem~\ref{thm:small_distance} has a perhaps more natural description: Consider the Baumslag-Solitar group $BS(1,p)=\langle a,x|x^{-1}ax=a^p\rangle$ and build the HNN-extension with associated subgroups $\langle a\rangle$ and $\langle x\rangle$. The resulting group is isomorphic to $G_p$. Furthermore, if we replace the $p$ in the presentation by the number $2$ we obtain what is called the Baumslag-Gersten group $G_2$. This group was constructed by Gersten~\cite{Gerst} (see also~\cite{Plat}) as an example of a group with Dehn function $\sim\, ^n2$.
\end{remark}
\begin{remark}
From now on we consider $p\geq 20$ to be a fixed number. We chose a lower bound of $20$  for the sake of brevity of the arguments. However, this is not the best possible bound for $p$. We believe the theorem to hold for all $p\geq 2$.
\end{remark}

We already remarked that the remainder of this section and the following two sections are devoted to the somewhat lengthy proof of Theorem~\ref{thm:small_distance}. The main aim of the rest of the present section is to introduce certain short geodesic words $w_k$ of $G_p$, which represent large powers of $a$. The words $w_k$ will later be used to show that $t(a^\infty,a^{-\infty})$ is bounded from above by $\sqrt{12/17}$. 

\medskip

For the sake of simplicity, let us shift our attention for a moment from $G_p$ to the infinitely generated group $G'$ that shall be defined next. First, for all $i<k\in\ZZ$ set 
\begin{eqnarray*}
 A_i^k &:=& \{a_i, \ldots, a_k\}, \\
 A_i^\infty &:=& \bigcup_{k\geq i}A_i^k,\\
 G_i^k &:=& \langle A_i^k\mid a_j^{-1}a_{j-1}a_j=a_{j-1}^p, j=i+1,i+2,\ldots ,k \rangle ,\\ 
 G_i^\infty &:=& \langle A_i^\infty \mid a_j^{-1}a_{j-1}a_j=a_{j-1}^p, j=i+1,i+2,\ldots \rangle,\textrm{ and }\\
 G'&:=&G_0^\infty.
\end{eqnarray*}

For all $i\in\ZZ$ the monoid isomorphisms $\varphi=\varphi_i:(A_i^\infty)^*\rightarrow (A_{i+1}^\infty)^*$ defined by  $\varphi(a_j)= a_{j+1}$ induce isomorphisms between $G_i^\infty$ and $G_{i+1}^\infty$ resp. $G_i^k$ and $G_{i+1}^{k+1}$, which we, abusing notation, will also call $\varphi$.
 Using $|i|$ times this isomorphism $\varphi$ we see that
\begin{equation}\label{isomorphsubgrps}
 \text{ $G_i^\infty\cong G'$ and  $G_j^k\cong G_{j+i}^{k+i}$ for all $i,j,k\in\mathbb N$. }
\end{equation}

\begin{lemma}\label{lemma:subgrpsembed}
Let $j<i\leq k$. In the notation defined above  $G_i^k = \langle A_i^k \rangle_{G_j^\infty}$ and $ G_i^\infty =\langle A_i^\infty\rangle_{G_j^\infty}$ (or, to be more precise, the identity map from $A_i^k\subset G_i^k$ to $A_i^k\subset G_j^\infty$ (resp. $A_i^\infty$) induces an isomorphism $G_i^k \cong \langle A_i^k \rangle_{G_j^\infty}$ (resp. $ G_i^\infty \cong\langle A_i^\infty\rangle_{G_j^\infty}$)). 
\end{lemma} 
\begin{proof}
As a first step for fixed $i$ we use induction on $k$ to show that $G_i^k=\langle A_i^k\rangle_{G_j^k}$. Let $k=i$, then $G_i^k\cong\mathbb{Z}$ and in the HNN-extension $G_j^k$ the letter $a_k$ is the stable letter and hence $\langle A_i^k\rangle_{G_j^k}=\langle a_k\rangle_{G_j^k}\cong\mathbb{Z}$. Now assume  $G_i^m=\langle A_i^m\rangle_{G_j^m}$ and let $k=m+1$. By von Dyck's theorem the identity map on $A_i^k$ induces an epimorphism $id:G_i^k\rightarrow\langle A_i^k\rangle_{G_j^k}$ and we only have to check for injectivity. The group $G_i^k$ is an HNN-extension with stable letter $a_k$ and base group $G_i^m$ and similarly $G_j^k$ is an HNN-extension with stable letter $a_k$ and base group $G_j^m$. The induced epimorphism maps words in normal form to words in normal form, so the injectivity is a consequence of Britton's lemma.  

It remains to show that $G_j^k=\langle A_j^k\rangle_{G_j^\infty}$. Again von Dyck's theorem shows that the identity map induces an epimorphism and we just have to check injectivity.  Let $w\in (A_j^k)^*$ and assume $w=_{G_j^\infty}1$. Then $w$ is freely equivalent to a finite product of relators and therefore there exists an $m$ such that $w=_{G_j^m}1$. The group $G_j^m$ is an iterated HNN-extension of $G_j^k$ and in each step injectivity is an immediate consequence of Britton's lemma. Hence $w=_{G_j^k}1$.

The claim about $G_i^\infty$ being isomorphic to $\langle A_i^\infty\rangle_{G_j^\infty}$ follows from the observation that $G_i^\infty=\underrightarrow{\mathop{lim}} G_i^k\cong \underrightarrow{\mathop{lim}} \langle A_i^k\rangle_{G_j^\infty}=\langle A_i^\infty\rangle_{G_j^\infty}$.
\end{proof}
The case of the lemma above one should keep in mind is the case $0=j<i$, hence $G_j^\infty=G'$. We only have to deal with negative values of $i$ for some technical reasons but will see later on (in Lemma~\ref{containingnoai}) that the letters $a_i$ for negative $i$ are of no importance for our purposes.
 
We shall now embed $G'$ in $G_p$.
By~\eqref{isomorphsubgrps}, the subgroup generated by the elements $\{a_i,a_{i+1},a_{i+2}\ldots\}$ is isomorphic to $G'$. Therefore we can construct the ascending HNN-extension $G$ associated to $\varphi$. Then
\[G=\langle t, a_j (j=0,1,2\ldots)\mid a_{j+1}^{-1}a_ja_{j+1}=a_j^{p},t^{-1}a_jt=a_{i+1}\rangle.\]
Note that in this group the relations $a_i=t^{-i}a_0t^i$ hold. Substituting $a_0$ by $a$ and applying Tietze-transformations we obtain the presentation from Theorem~\ref{thm:small_distance}:
\[G=G_{p}=\langle a,t\mid t^{-1}a^{-1}tat^{-1}at=a^{p}\rangle .\]
So $G$ is in fact a one-relator group on two generators. Even if the elements $a_i$ no longer belong to our set of generators, we will still use the notation $a_i$ for the element $t^{-i}at^i$.  In order to prove Theorem~\ref{thm:small_distance} we are only interested in distances between powers of $a$, hence elements of the subgroup $G'$. Such words have to contain the same number of letters $t$ and $t^{-1}$. Moreover, they can be written entirely in letters $a_i$ using the following rewriting process:

Let $v$ be a word in $\{a^\pm,t^\pm\}^*$ as above. We replace every $a$ by the letter $a_i$ and every $a^{-1}$ by $a_i^{-1}$, where $i$ is the difference of the number of $t^{-1}$'s and the number of $t$'s before this $a$ or $a^{-1}$, respectively. Afterwards we delete all letters $t^\pm$ to obtain the word $v'\in\{a_i^\pm\}_{i\in\ZZ}^*$. For example $v=t^{-2}at^4a^2t^{-3}a^{-5}ta$ becomes $v'=a_2(a_{-2})^2(a_1)^{-5}a_0$. 

If the word $v$ is (freely) reduced, we can recover it from $v'$ by replacing each $a_i$  with $t^{-i}at^i$ and each $a_i^{-1}$ with $t^{-i}a^{-1}t^i$, respectively, and freely reducing the result then. This defines a bijection $\psi$ between the reduced words in $\{a_i^\pm\}_{i\in\ZZ}$  and  the reduced  words in $\{a^\pm,t^\pm\}^*$ that have  the same number of letters $t$ and $t^{-1}$.\label{defofpsi}

\medskip

We proceed to defining the words $w_k$ which shall be used as `shortcuts' to go from large negative powers to large positive powers of $a$ in the proof of Theorem~\ref{thm:small_distance}. Our definition of the $w_k$ will rely on the words $w'_k$ in $G'$ representing large powers of $a_0$ which we define first.

For this, first note that
\begin{align*}
a_{i+1}^{-k}\, a_i\, a_{i+1}^k 
&=_{G'}  a_{i+1}^{-(k-1)}\, a_i^p\, a_{i+1}^{k-1}\\[8pt]
&=_{G'} (a_{i+1}^{-(k-1)}\, a_i\, a_{i+1}^{k-1})^{p}\\[8pt]
&=_{G'} ((a_{i+1}^{-(k-2)}\, a_i\, a_{i+1}^{k-2})^p)^p\\[8pt]
&=_{G'} (a_{i+1}^{-(k-2)}\ a_i\ a_{i+1}^{k-2})^{p^2}\\[8pt]
&=_{G'}  \ldots\\[8pt]
&=_{G'} a_i^{p^k}.
\end{align*}

Now set $w_0':=a_0$ and 
inductively set $w_k':=\varphi(w_{k-1}')^{-1}a_0\varphi(w_{k-1})$. 
Notice that the word $w_k'$ only consists of $2^{k+1}-1$ letters. Nevertheless it represents a huge power of $a_0$:
\begin{lemma}\label{lemma:hugepowers}
The word $w_k'$ ($\in (A_0^\infty)^*$) is freely reduced and  represents the group element $a_0^\tet{k}$ in $G'$.
\end{lemma}
\begin{proof}
We use induction on $k$. For $k=0$ the statement is true by definition. Assume that  $w_n'$ is freely reduced, $w_n'=_{G'}a_0^\tet{n}$ and let $k=n+1$. The word $\varphi(w_n')$ is also freely reduced and does not contain the letter $a_0$ hence  $w_k':=(\varphi(w_n'))^{-1}\ a_0\ \varphi(w_n')$ is also freely reduced.

We obtain:
\begin{align*}
w_k'&:= (\varphi(w_n'))^{-1}\ a_0\ \varphi(w_n')\\
&=_{G'} (\varphi(a_0^\tet{n}))^{-1}\ a_0\ \varphi(a_0^\tet{n})\\
&=_{G'} a_1^{-\tet{n}}\ a_0\ a_1^{\tet{n}}\\
&=_{G'} a_0^\tet{n+1}\\
&=_{G'} a_0^\tet{k}\\
\end{align*}
\end{proof}

Since $w_i'$ is reduced we finally can define $w_i:=\psi (w'_i)$. Then:
%
\begin{eqnarray*}
w_0&=&a\ \textrm{ and }\\
w_{i+1} &=& t^{-1}w_{i}^{-1} t\  a\ t^{-1} w_{i} t,
\end{eqnarray*}
where the second line follows from the easy observation, that for all reduced words $w\in (A_0^\infty)^*$ the word $\varphi(w)$ is also reduced and  $\psi(\varphi(w))=t^{-1}\psi(w)t$. Note that since $\psi|_{G'}$ is just a rewriting process of elements of $G'$ as a subgroup of $G$ and $w_k'=_{G'}a^\tet{k}$ we obtain $w_k=_G a^\tet{k}$. 

The recursion formula for $w_k$ above implies that the  length of $w_k$ is given by the recursion formula $\ell(w_{i+1})=2\cdot \ell(w_i)+5$ and therefore 
\begin{eqnarray}\label{lengthwk}
\ell(w_k)&=& 3\cdot 2^{k+1}-5.
  \end{eqnarray}
Our proof of Theorem~\ref{thm:small_distance} will follow from the next two lemmas.

\begin{lemma}\label{coro}
 The words $w_k$ are geodesic.
\end{lemma}

Lemma~\ref{coro} will be proved in Section~\ref{sec:lemmas}.

The second key ingredient in the proof of Theorem~\ref{thm:small_distance} is Lemma~\ref{lemma:final}, to be stated next, and to be proved in Section~\ref{sec:final}.
We employ the well-known Kronecker delta $\delta_{m,n}$, which, here for numbers $n,m\in\mathbb Z[\frac 12]$, takes the value $1$ if $m=n$, and $0$ otherwise.

\begin{lemma}\label{lemma:final}
 Let $k>0$, $n\in\ZZ$ be such that $d(1,a^ n)=: d_n<3\cdot 2^{k+1}-5$. Then
 \[
n<{p^{p^{\cdot^{\cdot^{\cdot^{p^{12}}}}}}}
\]
where the number of $p$' s is $k-1$ and
\[
 d(1,a^{\tet{k}-n})\geq 3\cdot 2^{k+1}-5+\min\{d_n,3\cdot 2^{k}-5\}-(1-\delta_{k,1})\min\{d_n,2^{k-1}\}.
\]
\end{lemma}

Postponing the proofs of Lemma~\ref{coro} and  Lemma~\ref{lemma:final} to the next two sections we first show how they imply Theorem~\ref{thm:small_distance}:

\begin{proof}[Proof of Theorem~\ref{thm:small_distance}]
Observe that it suffices to show that for all $\alpha > 12/17$ there is a $c$ such that the elements $a^{-n}$ are contained in the $(\alpha,c)$-cones of $a^\infty$. Then by symmetry (interchanging $a$ and $a^{-1}$ in all arguments), the reciprocal is true as well, showing that the distance between $a^\infty$ and $a^{-\infty}$ is at most $\sqrt{12/17}$. Let 
$\alpha>12/17$ 
and set $c:=35/17$.
 
Let $n>0$.
Now, let $k=k(n)$ be the unique positive integer such that 
\begin{equation*}\label{defk}
3\cdot2^{k+1}-5>d(1,a^n)\geq 3\cdot 2^{k}-5. 
\end{equation*}
 We define $h=h(n):=\tet{k}-n$, which is according to Lemma~\ref{lemma:final} positive.
Hence, by Lemmas~\ref{lemma:hugepowers} and~\ref{coro}  and by~\eqref{lengthwk},
\begin{equation}\label{irgendwas}
d(a^{-n},a^{h})=d(1,a^{\tet{k}})= 3\cdot 2^{k+1}-5.
\end{equation} 
Using Lemma~\ref{lemma:final} we obtain
\begin{equation*}\label{claim}
 d(1,a^{h})>3\cdot (2^{k+1}+2^{k})-2^{k-1}-10.
\end{equation*}
By~\eqref{irgendwas} this shows that
\begin{eqnarray*}
  d(a^{-n},a^{h}) & = & 3\cdot 2^{k+1}-5 \\
& =& 12/17 \cdot (3\cdot (2^{k+1}+2^{k})-2^{k-1}-10)+120/17-5\\
& < &\alpha\cdot (3\cdot (2^{k+1}+2^{k})-2^{k-1}-10)+35/17\\
        &< &\alpha d(1,a^{h})+c,
\end{eqnarray*}
and thus $a^{-n}$ lies in the $(\alpha, c)$--cone around $a$.
\end{proof}

\section{The words $w'_k$ and $w_k$ are geodesic}\label{sec:lemmas}
The main aim of this section is to prove Lemma~\ref{coro}, namely that the words $w_k$ are geodesic in $G$. This will be obtained by a series of results on the groups $G^k_i$ and $G^\infty_i$. A bit outside our way towards Lemma~\ref{coro}, we will also sketch a proof for the fact that the words $w'_k$ are geodesic in $G'$ (Lemma~\ref{wstrichk}). 

The other important results of this section will be Lemmas~\ref{containingakinG} and~\ref{noakisshort} which are used in the proof of our main theorem, Theorem~\ref{thm:small_distance}. We start by showing a number of rather easy lemmas. Recall  that on page~\pageref{defofpsi} we defined a bijection $\psi$ between the reduced words in $\{a_i^\pm\}_{i\in\mathbb N}$  and  the reduced  words in $\{a^\pm,t^\pm\}^*$ that have  the same number of letters $t$ and $t^{-1}$. We will use the following notation: We say a word $w\in A^j_i$ is \emph{pseudo-geodesic} in $G'$ if $w$ is geodesic in $G_i^j$ or $\psi(w)$ is geodesic in $G$. 

\begin{lemma}\label{containingnoai}
  Let $i\in\NN$ and $k>i$ or $k=\infty$ and let $w'$ be a pseudo-geodesic word in $G'$. Any subword $w$ of $w'$ representing an element $g\in G_i^k$ is an element of $(A_i^\infty)^*$. 
\end{lemma}
\begin{proof}
Fix a $j<0$ such that $j<\min\{m: w'\text{ contains the letter } a_m^\pm\}$.
By Lemma~\ref{lemma:subgrpsembed}, the identity map on $A_i^\infty$ induces an embedding of $G_i^k$ into $G_i^\infty$ and of $G_i^\infty$ into $G_j^\infty$. Therefore $G_j^\infty$ splits as semi-direct product $G_j^\infty=\langle\langle A_j^{i-1}\rangle\rangle_{G_j^\infty}\rtimes G_i^\infty$, and hence $G_i^\infty=G_j^\infty/\langle\langle A_j^{i-1}\rangle\rangle_{G_j^\infty}$. 

Assume that $w$ contains at least one letter $a_m^\pm$ for $k\leq m< i$. Let $\pi:(A_j^\infty)^*\rightarrow (A_i^\infty)^*$ be the canonical projection, that is, $\pi(w)$ is the word we obtain by removing from $w$ all letters $a_m^\pm$ for $k\leq m< i$.

Then $\pi(w)=_{G_i^\infty}g\cdot\langle\langle A_j^{i-1}\rangle\rangle$ and, since $\pi(w)$ does not contain any letters $a_m^\pm$ for $0\leq m< i$, we obtain that $\pi(w)=_{G_0^\infty} g$. 
Replacing the subword $w$ in $w'$ by $\pi(w)$ we obtain a the shorter word $w''$. So $w''$ is shorter than $w$ and since $w=_{G_j^\infty}\pi(w)$ we also obtain $w'=_{G_j^\infty}w''$. Hence $w'$ is not geodesic. Moreover, all letters of $w''$ also occur in the same order in $w'$, only some more letters are inserted in between. Hence $\psi(w'')$ contains at most the same number of $t^\pm$s and less $a^\pm$, which implies $\ell(\psi(w''))<\ell(\psi(w'))$.

This contradicts the assumption that $w'$ is pseudo-geodesic in $G'$. 


\end{proof}
\begin{corollary}
The subgroups $G_i^\infty$ are undistorted in $G'$. That is, for the generators considered above, the distances between elements of $G_i^\infty$ are the same in $G_i^\infty$ as in  $G'$.
\end{corollary}

\begin{lemma}\label{lem:prodofconj}
Let $k>i$ and let $w$ be a pseudo-geodesic word in $G^k_i$ with $w=_{G^k_i}a_i^n$. Then there are words $v_\alpha$, $\alpha=1,\ldots m$, in $G^k_{i+1}$ such that
\begin{enumerate}[(a)] 
\item $w=a_i^{l_0}v_1a_i^{l_1}v_2a_i^{l_2}\ldots v_ma_i^{l_m}$, for some $l_j\in\ZZ$ with  $l_1,\ldots,l_{m-1}\neq 0$
\item $v_\alpha=_{G_{i+1}^k}a_{i+1}^{\beta_{\alpha}}$ for some $\beta_\alpha\in\ZZ$, and 
\item $\prod_{\alpha =1}^mv_\alpha=_{G^k_{i+1}}1$.
\end{enumerate}
\end{lemma}

\begin{proof}
Without loss of generality we may assume that $w$ does not end with a letter $a_i^\pm$. 
 
As a consequence of Lemma~\ref{containingnoai}, the word $w$ has the property that none of its subwords representing some element of some subgroup $G_l^k$ for $l>i$ may contain a letter $a_j^\pm$ for $j<l$. For fixed $i$, we use induction on $k$ to show the stronger statement that all words with this property that represent $a_i^n$ are of the desired form. 

For $k=i+1$ the group $G_i^k$ is the Baumslag-Solitar group $BS(1,p)$. Now the $v_\alpha$ are just powers of $a_{i+1}$ and the above statement breaks down to the immediate consequence of Britton's lemma that elements of the base group have exponent sum $0$ in the stable letter.

So suppose $k>i+1$, and assume the statement true for $k-1$. The group $G_i^k$ is an HNN-extension of $G_i^{k-1}$ with associated subgroups $\langle a_{k-1}\rangle$ and $\langle a_{k-1}^p\rangle$ and stable letter $a_k$. As $wa_i^{-n}=_{G^k_i}1$, Britton's lemma implies that $w$ contains a subword  $a_k^{-1}va_k$ or $a_kva_k^{-1}$, where $v$  represents an element of $\langle a_{k-1}\rangle$ or $\langle a_{k-1}^{p}\rangle$, respectively. (In particular, $v$ does not contain any letters $a_j^\pm$ for $j<k-1$.) Replacing any such subword $a_k^{-1}va_k$ by $a_{k-1}^{pl}$  or $a_kva_k^{-1}$ by $a_{k-1}^l$, respectively, for some suitable~$l$, we obtain a word with less occurrences of $a_k$ which still represents $a_i^n$. Repeating this procedure as long as there are letters $a_k$ in our word, we arrive at a word $w'\in (A_{i}^{k-1})^*$, which still represents $a_i^n$.

We wish to apply the induction hypothesis to $w'$, so we have to check if $w'$ contains any subword representing some element of some subgroup $G_l^{k-1}$ for $l>i$ and containing the letter $a_j^\pm$ for some $j<l$. Assume that $w'$ contains such a subword $u=_{G_l^{k-1}}g\in G_l^{k-1}$. Since multiplication with letters $a_{k-1}^\pm$ from the left or right does not change the desired properties of this word, we may  assume, without loss of generality, that $u$ does not start or end with a letter $a_{k-1}^\pm$. Since the replacement procedure described above only creates letters $a_{k-1}^\pm$ this implies that all replacements have been made either outside of $u$ or completely inside of $u$. In particular by undoing these replacements we can identify a subword $v'$ of $w$ with the properties that $v'=_{G_l^{k-1}}g\in G_l^{k-1}$ and $v'$ contains $a_j^\pm$ (since we did not add any letters $a_j^\pm$ during our modification), which contradicts the assumptions on $w$.

By the induction hypothesis, $w'$ has the form $a_i^{l_0}v_1a_i^{l_1}v_2a_i^{l_2}\ldots v_m$ with $v_\alpha=_{G_i^{k-1}} a_{i+1}^{\beta_\alpha}$. Since all replacements have been made inside the words $v_\alpha$,  also $w$ has the desired form. The statement follows.
\end{proof}

\begin{lemma}\label{lem:prodofconj_extended}
In the situation (and notation) of Lemma~\ref{lem:prodofconj}, for all $1\leq j<m$ it holds $\sum_{\alpha=1}^j\beta_\alpha\leq 0$ and $$n=l_0+\sum_{\alpha=1}^ml_\alpha p^{-\sum_{j=1}^\alpha\beta_j}.$$
\end{lemma}
\begin{proof}
According to Lemma~\ref{lem:prodofconj} the word $w=a_i^{l_0}v_1a_i^{l_1}v_2a_i^{l_2}\ldots v_ma_i^{l_m}$ which represents the same element as
$$\tilde{w}:=a_i^{l_0}\prod_{\alpha=1}^{m}\left(\prod_{j =1}^\alpha v_j \right) a_i^{l_\alpha} \left(\prod_{j =1}^\alpha v_j \right)^{-1}.$$
Since $v_j =_{G_i^\infty}a_{i+1}^{\beta_j }$ we obtain
\begin{equation}\label{eq:prod_i_i+1}
w=_{G_i\infty}\tilde{w}':= a_i^{l_0}\prod_{\alpha=1}^{m}\left(\prod_{j =1}^\alpha a_{i+1}^{\beta_j }\right) a_i^{l_\alpha} \left(\prod_{j =1}^\alpha a_{i+1}^{\beta_j }\right)^{-1}
\end{equation}
which we can analyze in the subgroup $G_i^{i+1}$, the Baumslag-Solitar group $BS(1,p)$. Recall that in this group all conjugate of $a_i$ by powers of $a_{i+1}$ commute.
If $\sum_{j=1}^\alpha\beta_j < 0$ we already know that
\begin{equation} \label{eq:is_power}
\left(\prod_{j =1}^\alpha a_{i+1}^{\beta_j }\right) a_i^{l_\alpha} \left(\prod_{j =1}^\alpha a_{i+1}^{\beta_j }\right)^{-1} =_{G_{i}^{i+1}} a_i^{l_\alpha p^{-\sum_{j=1}^j\beta_j}}.
\end{equation}
Since $\tilde{w}'=_{G_i^{i+1}}a_i^n$ this implies that
$$
\prod_{\alpha \text{ such that }\\ \sum_{\alpha=1}^j\beta_\alpha>0}\left(\prod_{j =1}^\alpha a_{i+1}^{\beta_j }\right) a_i^{l_\alpha} \left(\prod_{j =1}^\alpha a_{i+1}^{\beta_j }\right)^{-1} =_{G_i^{i+1}} a_i^{n'}
$$
for some $n'\in\ZZ$. According to Britton's Lemma this is only possible, if one of the $l_\alpha$ is a multiple of $p$ (which is equivalent to the statement that $a_i^{l_\alpha}\in\langle a_i^p\rangle$). But $w$ is pseudo-geodesic which obviously implies $l_\alpha<p$. Hence no such $\alpha exists$ and using \eqref{eq:is_power} to sum up \eqref{eq:prod_i_i+1} the statement follows.
\end{proof}


\begin{lemma}\label{containingak}
Let $k\geq i\geq 0$. Any geodesic word in $G'$ containing the letter $a_k^\pm$ and representing an element of $\langle a_{k-i}\rangle$ has length at least $2^{i+1}-1$.
\end{lemma}

\begin{proof}
Let $v$  be a geodesic word in $G'$ representing an element of $\langle a_{k-i}\rangle$. We prove the statement by induction on $i$. Let $i=0$. A word containing $a_k^\pm$ has at least length $1=2^{0+1}-1$. Now assume the statement to be true for $i=n-1$.

Let $v$ be a geodesic word representing $a_{k-n}^l$ and containing the letter $a_k^\pm$. According to Lemma~\ref{lem:prodofconj}, $v=a_{k-n}^{l_0}v_1a_{k-n}^{l_1}v_2\ldots v_ma_{k-n}^{l_{m}}$ where each $v_\alpha=_{G'}a_{k-n+1}^{\beta_\alpha}$ for some $\beta_\alpha$ and the product $v_1v_2\ldots v_{m}=1$. Since $v$ contains a letter $a_k^\pm$, there exists an $\alpha$, such that $v_\alpha$ contains $a_k^\pm$. Since $v_\alpha$ is geodesic, the induction hypothesis gives that $v_\alpha$ has length at least $2^{n}-1$. Since $v'=(\prod_{\gamma=\alpha+1}^{m}v_\gamma)(\prod_{\gamma=1}^{\alpha-1}v_\gamma)$ is a word representing $v_\alpha^{-1}$ this word cannot be shorter than the geodesic word $v_\alpha$ and also contains at least $2^n-1$ letters. All in all, since $v$ contains at least $1$ letter $a_{k-n}$ we obtain that the length of $v$ is at least $2\cdot(2^n-1)+1=2^{n+1}-1$.
\end{proof}

In particular the last lemma shows that there exists no geodesic word containing $a_k^\pm$ and representing an element of $\langle a_{0}\rangle$, which is shorter than $w_k'$. And in fact the following lemma, which will not be needed in the course of this paper, holds:

\begin{lemma}\label{wstrichk}
The word $w_k'$ is a geodesic word in $G'$.
\end{lemma}
\begin{proof}
The word $w_k'$ represents the element $a_0^\tet{k}$ and has length $2^{k+1}-1$. So, by Lemma~\ref{containingak} for $i=k$ we only have to show that every geodesic word representing $a_0^\tet{k}$ has to contain the letter $a_k$. This can again be done by induction on $k$. The statement is obviously true for $k=0$. Because we won't need this statement later on, we leave the proof of the induction step, which can be done following the lines of the proof to Lemma~\ref{noakisshort}, to the reader.
\end{proof}

In contrast to the situation in $G'$ the product $w_iw_j$ for $i\neq j$ is not freely reduced. Nevertheless in the group $G$ the analogue of Lemma~\ref{containingak} also holds. 

\begin{lemma}\label{containingakinG}
Let $k\geq 0$. Let $w$ be a geodesic word in the letters $\{a,t\}$ representing a non-zero power of $a$ such that $w'=\psi^{-1}(w)$ contains the letter $a_k^\pm$. Then the length of $w$ is at least $3\cdot 2^{k+1}-5$.

If in addition $\ell(w)=3\cdot 2^{k+1}-5$, then $w=_G a^{\pm(\tet{k})}$.
\end{lemma}

\begin{proof}
Without loss of generality we may assume that $k=\max\{j: a_j^\pm\textrm{ is contained in } w'\}$. We prove the statement by induction on $k$. For $k=0$ the statement is trivial.

For $k>0$, Lemma~\ref{lem:prodofconj} yields that $w'=a_0^{l_0}v'_1a_0^{l_1}v'_2a_0^{l_2}\ldots v'_ma_0^{l_m}$ where each $v'_\alpha=_{G'}a_{1}^{\beta_\alpha}$ for some $\beta_\alpha$ and the product $v'_1v'_2\ldots v'_{m}=1$. 
The $v'_\alpha$ are subwords of the pseudo-geodesic word $w'$ and according to Lemma~\ref{containingnoai} do not contain any letters $a_0$.
Then, for some words $v_\alpha\in\{a,t\}^*$ we obtain: $$w=a^{l_0}t^{-1}v_1ta^{l_1}t^{-1}v_2t a^{l_2}\ldots a^{l_{m-1}}t^{-1}v_mta^{l_m}$$  where each $v_\alpha=_{G}a^{-\beta_\alpha}$ and the product $v_1v_2\ldots v_{m}=1$ (note that this is the same as saying that $\sum\beta_\alpha=0$). Since $w$ is geodesic, $\beta_\alpha\neq 0$ for all $\alpha$ which immediately implies $m\geq 2$. Therefore the number of $t$'s or $t^{-1}$'s outside of the $v_\alpha$ is at least $4$.

Since $w'$ contains a letter $a_k^\pm$, there exists an $\alpha^*$, such that $\psi^{-1}(v'_{\alpha^*})$ contains $a_{k-1}^\pm$. As a subword of $w$, the word $v_{\alpha^*}$ is geodesic, it has by induction hypothesis length at least $3\cdot 2^{k}-5$. Because $v_1v_2\ldots v_{m}=1$ the product of the other $v_\alpha$ also has length at least $3\cdot 2^{k}-5$, and furthermore, we have at least four $t$'s and an $a^{l_\alpha}$, the bound follows.

For the second assertion of the lemma, we again apply induction on $k$. The case $k=0$ is trivial. So assume the statement correct for $k-1$. The $v_{\alpha^*}$ defined above has -- according to the first part of this lemma -- length at least $3\cdot 2^k-5$. Since $\Pi_{\alpha\neq\alpha^*}v_\alpha=v_{\alpha^*}$ and $v_{\alpha^*}$ is geodesic, we obtain that also $\Pi_{\alpha\neq\alpha^*}v_\alpha$ contains at least $3\cdot 2^k-5$ letters. In addition $w$ contains at least 4 letters $t^\pm$ and one $a^\pm$. This only works out if $w=t^{-1}v_1ta^\pm t^{-1}v_2t$ and $\ell(v_1)=\ell(v_2)=3\cdot 2^k-5$. Since $w$ represents a power of $a$ we obtain $wa^x=_G1$ for some (huge) $x\in\ZZ$. Britton's lemma now implies that $v_1=v_2^{-1}$ has to be a power of $a$ and by induction hypothesis $v_1=_Ga^{-\tet{k-1}}$ and $v_2=_Ga^\tet{k-1}$. Hence
\begin{eqnarray*} 
w&=_G &t^{-1}a^{-\tet{k-1}}t\ a^\pm\ t^{-1}a^{\tet{k-1}}t\\
&=_G& a^{\pm\tet{k}}.
\end{eqnarray*}
\end{proof}

Furthermore we can bound the power of $a$ which is represented by a word of given length avoiding high powers of $t$.

\begin{lemma}\label{noakisshort}
Let $k,L\geq 1$. Let $v$ be a word  of length less than $L\cdot 2^{k-1}$ in $G$ representing an element $a^n$ for some $n\in\ZZ$ such that $\psi^{-1}(v)$ does not contain the letter $a_{k}^\pm$. Then, 
$$|n|< {p^{p^{\cdot^{\cdot^{\cdot^{p^{L}}}}}}}$$
for $k>1$ where the number of $p$'s is $k-1$ and $|n|<L$ for $k=1$.
\end{lemma}

\begin{proof}
Let $v'=\psi^{-1}(v)$. Without loss of generality we assume $v'$ to be reduced. First we show, that we also may assume $j:=\max\{\alpha: a_\alpha ^\pm\textrm{ is contained in } v'\}<k$. 

So assume that $j>k$. Since $v'a_0^{-n}=_{G'}1$, Britton's lemma implies that $v'$ contains a subword  $a_j^{-1}wa_j$ or $a_jwa_j^{-1}$, where $w$  represents an element of $\langle a_{j-1}^{p}\rangle$ or $\langle a_{j-1}\rangle$, respectively. Replacing any such subword $a_j^{-1}va_k$ by $a_{j-1}^{pl}$  or $a_jva_j^{-1}$ by $a_{j-1}^l$, respectively, for some suitable~$l$, we obtain a word with less occurrences of $a_j$ which still represents $a_0^n$. Repeating this procedure as long as there letters $a_j$ in our word, we arrive at a word, which still represents $a_i^n$ but does not contain $a_j^\pm$. We repeat this procedure with $a_{j-1}^\pm$ and all $a_\alpha ^\pm$ down to $\alpha=k+1$ and end up with a word $v''a_0^{-n}=_{G'}1$ that consists only of letters $a_0^\pm,\ldots a_{k-1}^\pm$ and $a_{k+1}^\pm$. This word contains no subword representing an element of $\langle a_k\rangle$ or $\langle a_k^p\rangle$, so all the $a_{k+1}^\pm$ have to freely cancel each other. The resulting reduced subword $v'''$ of $v'''a_0^{-n}$ can be obtained from $v'$ by deleting some of the letters of $v'$. Hence $\ell(\psi(v'''))<\ell(v)$, $\psi(v''')=_Gv$, and $j:=\max\{l: a_\alpha ^\pm\textrm{ is contained in } v'''\}<k$.

Therefore we may assume $j<k$ and $v'=\psi^{-1}(v)$ does not contain any $a_\alpha$ with $\alpha\geq k$.

We proceed by induction on $k$. Let $k=1$. The word $\psi^{-1}(v)$ does not contain a letter $a_1^\pm$. Therefore $v=a^\alpha$  for some $|\alpha|<L$. Obviously $n=\alpha$ and we are done.

Let  $k\geq 2$ and assume the statement to be true for $k-1$. We only consider the case that $n$ is positive, as the other case is symmetric. We may assume that $v$ is such that $n$ is maximal among all possible values for $n$ over all choices of $v$ as in the lemma.
Note that then  $\ell(v)=L\cdot 2^{k-1}-1$, and furthermore, $v$ is shortest possible among all $v$ satisfying the assumptions of the lemma.

Now, as in the proof of Lemma~\ref{containingakinG} we obtain
\[v=a^{l_0}t^{-1}v_1ta^{l_1}t^{-1}v_2t\ldots t^{-1}v_mta^{l_m},\]
with $v_i=_Ga^{-\beta_i}$ for some $\beta_i$ such that $\sum \beta_i=0$. But now we can calculate $n$ in terms of $l_i$ and $\beta_i$, namely 
\[
n= l_0+\sum_{i=1}^m l_ip^{\sum_{j=1}^i\beta_j}\leq ( \sum_{i=0}^m |l_i|)\cdot p^{\max_i\sum_{j=1}^i\beta_j}=:y.
\] 
Let $c$ be such that $\max_i\sum_{j=1}^i\beta_j=\sum_{j=1}^c\beta_j$. By deleting all but four letters $t$ and rearranging the letters $a$ we obtain the word 
\[
v'=t^{-1}v_1v_2\ldots v_cta^{\sum_{i=0}^ml_i}t^{-1}v_{c+1}\ldots v_mt.
\]
Then $\ell(v')\leq \ell(v)$ and $v'=_Ga^y$. Since $v$ was chosen such that $n$ is maximal, we obtain that $y=n$. Then $\ell (v)=\ell(v')$. So, we actually did not delete any $t$ when creating $v'$, and thus $v=v'$. Hence
\[
v=t^{-1}v_1ta^lt^{-1}v_2t,
\]
with $v^{-1}_1=_Gv_2=_Ga^\alpha$ for some $\alpha\geq 0$ and $n=l\cdot p^\alpha$.

Assume that $l\geq 3$. Then we can build the word $v''=t^{-1}v_1a^{-1}ta^{l-2}t^{-1}av_2t$ which is of the same length as $v$ and represents $a^{((l-2)p)p^\alpha}$ in contradiction to the maximality of $n$. Therefore $l\leq 2$. Since $v$ is shortest possible under the assumptions of the lemma so are $v_1$ and $v_2$, and hence $\ell(v_1)=\ell(v_2)$. Since $\ell(v)$ is odd, it follows that $l=1$ and 
\[\ell(v_1)=\frac{L\cdot 2^{k-1}-1-5}{2}<L\cdot 2^{k-2}.\]
By induction hypothesis
\[
|\alpha | < {p^{p^{\cdot^{\cdot^{\cdot^{p^{L}}}}}}}
\]
where the number of $p$'s is $k-2$ and since $n=1\cdot p^\alpha$ we obtain the desired inequality.
\end{proof}

The two preceding lemmas imply Lemma~\ref{coro}, that is, that  the $w_k$ are geodesic:

\begin{proof}[Proof of Lemma~\ref{coro}]
 Let $w$ be a geodesic word such that $w=_{G}w_k=_Ga^\tet{k}$. 
Recall that we have chosen $p\geq 20>12$. Since $\ell(w)<12\cdot 2^{k-1}$ Lemma~\ref{noakisshort} implies that $w$ has to contain the letter $a_k$ or $a_k^{-1}$. On the other hand, by Lemma~\ref{containingakinG} we know that any word containing $a_k$ or $a_k^{-1}$ is as least as long as $w_k$. So the statement follows.
\end{proof}

\section{The proof of Lemma~\ref{lemma:final}}\label{sec:final}

This final section is devoted to the proof of Lemma~\ref{lemma:final}, which is the only ingredient missing for our proof of Theorem~\ref{thm:small_distance}. We build on results from Section~\ref{sec:lemmas}. Before we start to consider the general situation let us focus on some cases of small values for $n$, which turns out not only to be more accessible but will also be of importance during the proof of the general case. For this case we actually need stronger statements:
\begin{lemma}\label{lem:n=p_case}
 For $k\geq 2$ we have $d(1,a^{\tet{k}-p})\geq 3\cdot 2^{k+1}$.
\end{lemma}
\begin{proof}
 Let $v=_G a^{\tet{k}-p}$ and set $L=12$. Then $\tet{k}-p>p^{p^{.^{.^{.^{12}}}}}$, where the number of $p$'s is $k-1$ For contradiction, assume $\ell(v)<3\cdot 2^{k+1}$. Then by Lemma~\ref{containingnoai}, the word $v':=\psi^{-1}(v)$ contains the letter $a_k$.

By Lemma~\ref{lem:prodofconj} we can write $v'$ as $a_0^{l_0}v_1a_0^{l_1}v_2a_0^{l_2}\ldots v_ma_0^{l_m}$, with $l_1,\ldots,l_{m-1}\neq 0$ and thus, 
\begin{equation}
v=a^{l_0}t^{-1}u_1ta^{l_1}t^{-1}u_2ta^{l_2}\ldots t^{-1}u_m ta^{l_m},
\end{equation}
for some $u_i=_G a^{n_i}$. Since one of the $v_i$ contains a letter $a_k$, one of the $u_i':=\psi^{-1}(u_i)$, say $u_j'$ contains a letter $a_{k-1}$. Hence, Lemma~\ref{containingakinG} implies $\ell(u_j)\geq 3\cdot 2^k-5$. Another consequence of Lemma~\ref{lem:prodofconj} is that 
\begin{equation}\label{eq:j=neqj}
 \prod_{i\neq j} u_i =_G u_j^{-1}
\end{equation}
 and since $u_j$ is geodesic $\sum_{i\neq j}\ell(u_i)\geq 3\cdot 2^k-5$. This already implies $m\leq 3$. Assume $m=3$. This implies $\ell({u_j})=3\cdot 2^k-5$, since otherwise we obtain $\ell(v)\geq 2\cdot(3\cdot 2^k -4) + 8$ (the word $v$ contains 6 additional $t,t^{-1}$'s and at least two additional $a$'s). By Lemma~\ref{containingakinG} this implies $u_j=_G a^{\pm\tet{k-1}}$. By \eqref{eq:j=neqj} and Lemma~\ref{containingnoai} we obtain that $\psi^{-1}(\prod_{i\neq j} u_i)$ also contains $a_{k-1}$. So for $j'\neq j$ the word $u_j'$ contains $a_{k-1}$. Since $u_j\neq u_j'^{\pm}$, Lemma~\ref{containingakinG} implies $\ell(u_j')\geq 3\cdot 2^k-4$ and therefore $\ell(v)\geq 3\cdot 2^k-5+3\cdot 2^k-4+1+8$. 

So, $m=2$ and $v=a^{l_0}t^{-1}u_1ta^{l_1}t^{-1}u_2ta^{l_2}$ and similar arguments as above show that $\ell(u_1)=\ell(u_2)\geq 3\cdot 2^{k}-5$. So $|l_0|+|l_1|+|l_2|<10$. By Lemma~\ref{lem:prodofconj_extended} we obtain $\tet{k}-p=l_0+l_1\cdot p^\delta+l_2$ for some $\delta$. But this is impossible since $p\geq 20 > 16$.
\end{proof}

\begin{lemma}\label{lem:small_n}
 Let $k\geq 2$ and $n\in\ZZ$ be such that no geodesic word representing $a^n$ contains the letter $t$ (which is easily seen to be equivalent to $|n|< \frac{p+7}{2}$), then  
\[
 d(1,a^{\tet{k}-n})\geq 3\cdot 2^{k+1}-5+|n|-\delta_{|n|,\frac{p+6}{2}}.
\]
\end{lemma}
\begin{proof}
  We only discuss the case $n>0$, the case $n<0$ can be shown analogously and the case $n=0$ is a consequence of Lemma~\ref{coro}. Let $v=_G a^{\tet{k}-n}$ be a geodesic word. By Lemma~\ref{lem:prodofconj} we can write $\psi^{-1}(v)$ as $a_0^{l_0}v_1a_0^{l_1}v_2a_0^{l_2}\ldots v_ma_0^{l_m}$, with $l_1,\ldots,l_{m-1}\neq 0$ and thus, 
\begin{equation}\label{eq:small_n_v=}
v=a^{l_0}t^{-1}u_1ta^{l_1}t^{-1}u_2ta^{l_2}\ldots t^{-1}u_m ta^{l_m}
\end{equation}
  where $u_i=_G a^{-\alpha_i}$, for $i=1,\ldots m$. Clearly, 
\begin{equation}\label{eq:small_n_thesum}
 \tet{k}-n=l_0+\sum_{i=1}^ml_i p^{\sum_{j=1}^i\alpha_j},
\end{equation}
 and the sum of all $\alpha_i$ equals $0$. Note that since $v$ is geodesic, we may assume that $\sum_{j=1}^i\alpha_j=0$ holds only for $i=m$ and $|l_i|<p$ for $i=1,\ldots m$. Considering equation~\eqref{eq:small_n_thesum} modulo $p$ we obtain $l_0+l_m\equiv -n(p)$, hence $l_0+l_m=-n$ or $l_0+l_m=p-n$. If $l_0+l_m=-n$, then, according to~\eqref{eq:small_n_thesum}, the subword $\tilde{v}:=t^{-1}u_1ta^{l_1}t^{-1}u_2ta^{l_2}\ldots t^{-1}u_m t$ of $v$ then has to represent the group element $a^{\tet{k}}$ and by Lemma~\ref{coro} we have $\ell(\tilde{v})\geq 3\cdot 2^{k+1}-5$, which implies the statement. 

Now assume $l_0+l_m=p-n$. Since $n<\frac{p+7}{2}$ we know that $p-n\geq n-5-\delta_{|n|,\frac{p+6}{2}}$. According to~\eqref{eq:small_n_thesum} the subword $\tilde{v}:=t^{-1}u_1ta^{l_1}t^{-1}u_2ta^{l_2}\ldots t^{-1}u_m t$ of $v$ then has to represent the group element $x=a^{\tet{k}-p}$. By Lemma~\ref{lem:n=p_case} we obtain $d(1,x)+l_0+l_m\geq 3\cdot 2^{k+1}+n-5-\delta_{|n|,\frac{p+6}{2}}$. 
\end{proof}

Now we are ready to prove the final missing piece of Theorem~\ref{thm:distances}:

\begin{proof}[Proof of Lemma~\ref{lemma:final}]
First of all, observe that by Lemma~\ref{containingakinG}, every geodesic word representing $a^n$ is $a_k$-less. So by Lemma~\ref{noakisshort}, we know that 
\begin{equation}\label{n<ppp}
n<{p^{p^{\cdot^{\cdot^{\cdot^{p^{12}}}}}}}
\end{equation}
where the number of $p$' s is $k-1$. 

In order to prove the lemma, we use induction on $k$. For $k=1$, we only have to check that $d(1,a^{p-n})\geq 7+\min\{n,1\} \geq 12-5+\min\{n,1\}-0$ for all $n$ with $n<12$, by~\eqref{n<ppp}. This is true, since testing all words with at most $6+\min\{n,1\}$ letters we see that none of them represents $a^{p-n}$ (note that by the choice of $p$, we have  $p-n>8$).

So assume the lemma valid for $k-1$, our aim is to show it for $k\geq 2$. Suppose otherwise, that is, assume there is a word $v$ with $v=_{G} a^{\tet{k}-n}$ and 
\begin{align}\label{ellv}
 \ell (v)< & \ 3\cdot 2^{k+1}-5+\min\{d_n,3\cdot 2^k-5\}-\min\{d_n,2^{k-1}\}\\ 
\leq & \ 3\cdot (2^{k+1}+2^k) -10\notag \\
\leq & \ 18\cdot 2^{k-1} -10.\notag
\end{align}
We claim that 
\begin{equation}\label{ctnsak}
 \psi^{-1}(v)\text{ contains the letter }a_k.
\end{equation}
 In fact, otherwise we may apply Lemma~\ref{noakisshort} to $v$, with $L=18$, to obtain that $$\tet{k}-{p^{p^{\cdot^{\cdot^{\cdot^{p^{12}}}}}}} <\tet{k}-n<{p^{p^{\cdot^{\cdot^{\cdot^{p^{18}}}}}}}$$ where on either side the number of $p$' s equals $k-1$, and the first inequality follows from~\eqref{n<ppp}. This, however, is  impossible, as $p\geq 20$. We have thus proved~\eqref{ctnsak}.

Now, by Lemma~\ref{lem:prodofconj}, we can write $\psi^{-1}(v)$ as $a_0^{l_0}v_1a_0^{l_1}v_2a_0^{l_2}\ldots v_ma_0^{l_m}$, with $l_1,\ldots,l_{m-1}\neq 0$ and thus, 
\begin{equation}\label{v=}
v=a^{l_0}t^{-1}u_1ta^{l_1}t^{-1}u_2ta^{l_2}\ldots t^{-1}u_m ta^{l_m}
\end{equation}
  where $u_i=_G a^{-\alpha_i}$, for $i=1,\ldots m$. Clearly, 
\begin{equation}\label{thesum}
 \tet{k}-n=l_0+\sum_{i=1}^ml_i p^{\sum_{j=1}^i\alpha_j},
\end{equation}
 and the sum of all $\alpha_i$ equals $0$. Note that since $v$ is geodesic, we may assume that $l_i<p$ for $i=1,\ldots m$.

Suppose $c\in\{1,\ldots,m\}$ is such that $\psi^{-1}(u_c)$ contains the letter $a_{k-1}^\pm$. Then by Lemma~\ref{containingakinG}, 
\begin{equation}\label{lengthuc}
 \ell(u_c)\geq 3\cdot 2^k -5.
\end{equation}
So, as $3\cdot (3\cdot 2^k -5)>\ell (v) -5$, and moreover, since each $u_c$ as above gives rise to two letters $t$, we conclude that
 there are less than $3$  indices $c$ such that  $\psi^ {-1}(u_c)$ contains the letter $a_{k-1}^\pm$. On the other hand,  by~\eqref{ctnsak}, there is at least one such index, say $c_1$.

Moreover, since the expression in~\eqref{v=} contains $m$ times a subword of the form $t^{-1}u_it$, and also at least $m-1$ letters $a$, we can use~\eqref{ellv} and~\eqref{lengthuc} to get that
\begin{equation}\label{m}
 m\leq\frac{\ell (v)-\ell (u_{c_1})+2}{4} < 3\cdot 2^{k-1}.
\end{equation}

Together,~\eqref{thesum} and~\eqref{m} imply that there is an index $b$ such that $$l_b \cdot p^{\sum_{j=1}^b\alpha_j}>\frac{{p^{p^{\cdot^{\cdot^{\cdot^{p^{p-1}}}}}}}}{3\cdot 2^{k-1}}$$ where the number of $p$'s equals $k-1$. Hence, since $p>6$, and since $l_b < p$, we know that  $$ p^{\sum_{j=1}^b\alpha_j}>\frac{{p^{p^{\cdot^{\cdot^{\cdot^{p^{p-1}}}}}}}}{p^{k}}$$ where again, the number of $p$'s is $k-1$. Taking the logarithm, we obtain that 
\begin{equation}\label{defy}
 x:=\sum_{j=1}^b\alpha_j> {p^{p^{\cdot^{\cdot^{\cdot^{p^{p-1}}}}}}}-k =:y
\end{equation}
 where  the number of $p$'s is $k-2$. Because $\sum_{j=1}^b\alpha_j=-\sum_{j=b+1}^m\alpha_j$, this yields that 
$$u_1u_2\ldots u_b=_Gu_{b+1}u_{b+2}\ldots u_m=_Ga^x.$$ So, by Lemma~\ref{noakisshort}, there is a second index $c_2$ such that  $\psi^{-1}(u_{c_2})$ contains the letter $a_{k-1}^\pm$. We may assume that $c_2>b\geq c_1$. Note that by what we said above, $c_1$ and $c_2$ are the only indices $c$ such that  $\psi^{-1}(u_{c})$ contains the letter $a_{k-1}^\pm$.

Consider the subword $$z:=t^{-1}u_{c_1}ta^{l_{c_1}}t^{-1}u_{c_1+1}\ldots t^{-1}u_{c_2}t$$ of $v$. By the choice of the $c_i$,
\begin{equation}\label{dasechtelengthz}
  \ell (z)\geq 2\cdot (3\cdot 2^ k-5) +5.
\end{equation}
 So, 
\begin{equation}\label{lengthz}
   \ell (v)- \ell (z)\leq 3\cdot 2^k-5.
\end{equation}

Set
\[
 u:=u_{c_1-1}^{-1}u_{c_1-2}^ {-1}\ldots u_1^{-1}u_m^{-1}u_{m-1}^{-1}\ldots u_{c_2+1}^{-1}
\]
and consider the word 
\[
v':=a^ {l_0} t^{-1} u_1 t a^{l_1}\ldots a^{l_{c_1-1}}t^{-1} u t a^{l_{c_2}}t^{-1}u_{u_{c_2+1}}t a^{l_{c_2+1}}\ldots t^{-1} u_m t a^{l_m}.
\]
Then $v'=_G a^q$ where 
\[q=l_0 + \sum_{i=1}^{c_1-1}l_i p^{\sum_{j=1}^i\alpha_j} + \sum_{i=c_2}^{m}l_i p^{\sum_{j=1}^i\alpha_j}.\]
Here we used the fact that $\sum_{j=1}^{i}\alpha_j= - \sum_{j=i+1}^{m}\alpha_j$.

By~\eqref{lengthz} and by the definition of $v'$, we know that $\ell (v')\leq 2\cdot (3\cdot 2^k-5)$, and moreover, since $\psi^{-1}(v')$ does not contain any letter $a_k^\pm$, for $k>2$ we obtain that $$|q|< {p^{p^{\cdot^{\cdot^{\cdot^{p^{12}}}}}}}$$ where the number of $p$'s is  $k-1$. For $k=2$ we obtain $|q|<p^7$ since in this case $\psi^{-1}(v')$ is even $a_1$ free and of length less than $7$.
Set
\begin{equation}\label{def_of_s}
 s:=\sum_{\gamma=c_1}^{c_2-1}l_{\gamma}p^{\sum_{j=1}^{\gamma}\alpha_j}=\tet{k}-n-q.
\end{equation}
Then, for $k>2$, 
\begin{equation}\label{s}
 \tet{k}-2 {p^{p^{\cdot^{\cdot^{\cdot^{p^{12}}}}}}} <s < \tet{k}+ 2 {p^{p^{\cdot^{\cdot^{\cdot^{p^{12}}}}}}},
\end{equation}
where the number of $p$'s on each side is again $k-1$ and respectively, for $k=2$,
\begin{equation}\label{s_k=2}
 \tet{2}-2p^7 <s < \tet{2}+ 2p^7,
\end{equation}

On the other hand, by~\eqref{lengthuc} and since $\ell (v)< 3\cdot (2^{k+1}+2^k)-10$ by~\eqref{ellv}, we have that 
\begin{equation}\label{sumlengthui}
       \sum_{i=c_1+1}^{c_2-1}\ell (u_i) <3\cdot 2^k-5,
 \end{equation}
and, for each of these indices $i$, we know that $\psi^{-1}(u_i)$ is $a_{k-1}$-free.  Therefore, for $k\geq 2$, the exponents of $p$ in the sum expression~\eqref{def_of_s} of $s$ differ less than $ {p^{p^{\cdot^{\cdot^{\cdot^{p^{12}}}}}}}$ where the number of $p$'s is $k-2$. For $k=2$ the same differences are less than $7=3\cdot 2^2-5$ since in this case $\psi^{-1}(\prod_{i=c_1+1}^{c_2-1} u_i)$ is even $a_1$ free and of length less than $7$.

We claim that this implies that 
\begin{equation}\label{eq:mittag}
 s=\tet{k}.
\end{equation}
In fact, for $k=2$ we know by~\eqref{defy} that one summand in \eqref{def_of_s} is divisible by $p^x$ for some $x>(p-1)-2$ and therefore by the argument above each summand is divisible by $p^{(p-2)-6}$. So, $s=\delta\cdot p^{p-8}$ and the only possible value for $s$ in the interval \eqref{s_k=2} is $\tet{2}$.

For $k>2$ we obtain by~\eqref{defy} that one summand in \eqref{def_of_s} is divisible by $p^x$ for some
$$x>p^{p^{\cdot^{\cdot^{\cdot^{p^{p-1}}}}}}-k$$ 
where the number of $p$'s is $k-1$.
Therefore, by the argument above each summand is divisible by $p^{x'}$ for some 
$$x'>{x-{p^{p^{\cdot^{\cdot^{\cdot^{p^{12}}}}}}}>p^{\cdot^{\cdot^{\cdot^{p^{p-2}}}}}}$$
where the number of $p$'s is $k-2$. So we can write 
$$s=\delta \cdot {p^{p^{\cdot^{\cdot^{\cdot^{p^{p-2}}}}}}} $$
where the number of $p$'s is $k-1$, and $\delta$ is some integer.
As the term after $\delta$ is greater than the length of the interval from~\eqref{s}, we know that the only possible value for $s$ is $\tet{k}$.
This proves \eqref{eq:mittag}.

Thus $\tet{k}=\sum_{i=c_1}^{c_2-1}l_ip^{\sum_{j=1}^i\alpha_j}$. Since all the $\sum_{j=1}^i\alpha_j$ are different (as $v$ is geodesic) and the $l_i$ are in $(0,p)$, basic arithmetics (a sum of products of powers of $p$ with numbers smaller than $p$ can only give a power of $p$ if there is only one summand, and the factor is $1$) imply that $l_{c_1}=1$ and $c_2=c_1+1$. Hence $z$ can be written as
\[
 z=t^{-1}u_{c_1}tat^{-1}u_{c_2}t.
\]

Taking the logarithm in~\eqref{eq:mittag}, this implies that $\sum_{i=1}^{c_1}\alpha_i=\tet{k-1}$. Hence,
$$u_{c_1}=_Ga^{-\tet{k-1}+\sum_{i=1}^{c_1-1}\alpha_i}$$
and
$$u_{c_2}=_Ga^{\tet{k-1}+\sum_{i=c_2+1}^{m}\alpha_i}.$$

We now apply the induction hypothesis with $n_1:=\sum_{i=1}^{c_1-1}\alpha_i$ in the role of $n$, which satisfies the assumptions as $a^{\sum_{i=1}^{c_1-1}\alpha_i}=_Gu_1u_2\ldots u_{c_1-1}$. We then apply the   induction hypothesis again with $n_2:=\sum_{i=c_2}^{m}\alpha_i$ in the role of $n$, which satisfies the assumptions as $a^{\sum_{i=c_2}^{m}\alpha_i}=_Gu_{c_2}u_{c_2+1}\ldots u_{m}$. This gives for $j=1,2$
\[
 \ell(u_{c_j})\geq  3\cdot 2^k -5 + \min\{d_{n_j}, 3\cdot 2^{k-1}-5\} - \min\{d_{n_j},2^{k-2}\}.
\]
So, as $v$ contains $3m-1$ letters $a$ and $t$ outside the $u_i$, we obtain
\begin{align}\label{bbb}
 \ell (v) \geq & \ \ell (u_{c_1})+\ell (u_{c_2}) +\ell (u_1u_2\ldots u_{c_1-1}) +\ell (u_{c_2}u_{c_2+1}\ldots u_{m}) +3m-1\notag \\
\geq & \ \ell (u_{c_1})+\ell (u_{c_2}) +d_{n_1}+d_{n_2}+3m-1\notag \\
\geq & \ 3\cdot 2^{k+1}+3m-11\notag \\
 & \  + \sum_{j=1,2}(\min\{d_{n_j}, 3\cdot 2^{k-1}-5\} - \min\{d_{n_j},2^{k-2}\}+d_{n_j}).
\end{align}

Observe that by~\eqref{ellv}, and since the term in the sum above is always non-negative, we get that 
\begin{equation}\label{dngeq2k-2}
d_n\geq 2^{k-1}.
\end{equation}

We claim that for $j=1,2$
\begin{equation}\label{mindj}
d_{n_j}\leq 3\cdot 2^{k-1}-5\text{ or }d_{n_{3-j}}=0.
\end{equation}
Indeed, suppose  $d_{n_1}> 3\cdot 2^{k-1}-5$. Then  by comparing~\eqref{ellv} with~\eqref{bbb}, we obtain that 
\begin{align*}
3\cdot 2^{k}-5 -2^{k-1} \geq & \ \min\{d_n,3\cdot 2^{k}-5\} - 2^{k-1}\\
\geq & \    3m-6 +\min\{d_{n_1},3\cdot 2^{k-1}-5\} -2^{k-2} +d_{n_1}\\ & \  +\min\{d_{n_2},3\cdot 2^{k-1}-5\}-d_{n_2} + d_{n_2}\\
\geq & \  3m-6+3 \cdot 2^{k-1} -5-2^{k-2} +3\cdot 2^{k-1}-5+1\\ & \ +\min\{d_{n_2},3\cdot 2^{k-1}-5\}\\
\geq & \ 3m-5+3\cdot 2^{k} -10 -2^{k-2}+\min\{d_{n_2},3\cdot 2^{k-1}-5\}.
\end{align*}
Therefore, since $m\geq 3$,
\[
-2^{k-1} \geq -1 -2^{k-2}+\min\{d_{n_2},3\cdot 2^{k-1}-5\},
\]
implying that $$1\geq 2^{k-2}+\min\{d_{n_2},3\cdot 2^{k-1}-5\}.$$ Hence $d_{n_2}=0$. In the same way we get that the assumption  $d_{n_2}> 3\cdot 2^{k-1}-5$ implies that $d_{n_1}=0$.
This proves~\eqref{mindj}.

\smallskip

Let us define a new word $\tilde v$ which is obtained from $v$ by replacing $z$ with $t^{-1}\tilde v_1^{-1}\tilde v_2t$, where the $\tilde v_i$ are geodesic words for $a^{n_i}$. That is,
\[
\tilde v:=a^{l_0}t^{-1}u_1t\ldots a^{l_{c_1-1}}t^{-1}\tilde v_1^{-1}\tilde v_2ta^{l_{c_2}}t^{-1}u_{c_2}\ldots t^{-1}u_mta^{l_m}. 
\]
Clearly, $\tilde v$ represents $a^{n}$.

First, suppose that both $\tilde v_i$ contain a letter $t$.
Note that then we may assume that  each of the $\tilde v_i$ starts with a $t^{-1}$. Hence, $d_n\leq \ell(\tilde v)-2$. Observe that also, $d_{n_j}>0$. Hence, by~\eqref{mindj}, $d_{n_j}\leq 3\cdot 2^{k-1}-5$.

By~\eqref{ellv} and by~\eqref{dngeq2k-2},
 $$\ell (v)< 3\cdot 2^{k+1} -5 + d_n -2^{k-1}.$$
Moreover, since
$$\ell(z)=\ell(u_{c_1})+\ell(u_{c_2})+5,$$
and by~\eqref{mindj}, we obtain

\begin{align*}
d_n\leq &\ \ell(\tilde v)-2 \\
\leq &\ \ell(v)+\underbrace{d_{n_1}+d_{n_2}+2}_{\leq\ell(t^{-1}\tilde{v_1}^{-1}\tilde v_2t)}-\ell(z)-2\\
< & \ 3\cdot 2^{k+1}-5 + d_n - 2^{k-1}+d_{n_1}+d_{n_2}\\
& -\sum_{j=1,2}\left(3\cdot 2^{k}-5+d_{n_j} - \min\{d_{n_j},2^{k-2}\}\right)-5\\
\leq &\ d_n-2^{k-1}+\sum_{j=1,2}2^{k-2}\\
\leq & \ d_n,
\end{align*}
a contradiction.

So we may assume  that one of $\tilde v_1,\tilde v_2$ does not contain a letter $t$, say $\tilde v_1$. Then it might not be true that $d_n\leq \ell(\tilde v)-2$. On the other hand, we can then use Lemma~\ref{lem:small_n}. Hence the last calculation becomes

\begin{align}\label{k=3}
d_n\leq &\ \ell(\tilde v)\notag \\
\leq &\ \ell(v)+d_{n_1}+d_{n_2}+2-\ell(z)\notag\\
< & \ 3\cdot 2^{k+1}-5 + d_n - 2^{k-1}+d_{n_1}+d_{n_2}\notag \\
& -\sum_{j=1,2}\left(3\cdot 2^{k}-5+d_{n_j}\right) - \min\{d_{n_2},2^{k-2}\}+\delta_{|n|,\frac{p+6}{2}}-5 +2\notag\\
\leq &\ d_n-2^{k-1}+2^{k-2}+\delta_{|n|,\frac{p+6}{2}} +2\notag\\
\leq & \ d_n-2^{k-2}+\delta_{|n_1|,\frac{p+6}{2}} +2\notag,
\end{align}
which yields a contradiction for $k>3$. For $k=3$ we deduce $n_1=\frac{p+6}{2}>7=3\cdot 2^2-5$. So by~\eqref{mindj}, $d_{n_2}=0$. So we can substitute the last two lines of the calculation above with
\begin{align*}
d_n&<\ d_n-2^2+1+2\\
&\leq\  d_n-1,
\end{align*}
which is also a contradiction.

So let $k=2$. Then $\ell(v)<3\cdot 2^3-5+7-2=3\cdot 2^3$, by~\eqref{dngeq2k-2}. Therefore $m\leq 3$. If $m=3$ we have $\sum_{i=1}^3 \alpha_i=0$ and hence $\alpha_{c_1}\neq \alpha_{c_2}$. So $\alpha_{c_i}=\pm\tet{2}$ and $\alpha_{c_{3-i}}\neq\tet{2}$. By Lemma~\ref{containingakinG} we get $\sum_{i=1,2}\ell(u_{c_i})\geq 3\cdot 2^3-9$ and
\[
\ell(v)\geq\sum_{i=1}^3\ell ({u_i})+3m-1\geq 3\cdot 2^3>\ell(v).
\]

So we have $k=m=2$. This implies $v=a^{l_0}t^{-1}u_{c_1}tat^{-1}u_{c_2}a^{-n+l_0}$ and $\ell(v)=3\cdot 2^3-5+n> 3\cdot 2^3-5 + d_n-1$, which is impossible by~\eqref{ellv}.
\end{proof}

\bibliographystyle{plain}
\bibliography{graphs}

\end{document}